\documentclass[a4paper]{article}
\usepackage{amsmath,amscd}
\usepackage{amssymb}
\usepackage{amsthm}
\usepackage{amsbsy,amstext,amsxtra,amsopn}
\makeatletter
  
  \@addtoreset{equation}{section}
\makeatother
\usepackage{latexsym}
\begin{document}
\title{{\bf Cohomology groups of sections of homogeneous line
bundles over a toroidal group}}
\author{Yukitaka Abe}
\date{ }
\maketitle

\noindent
{\bf Abstract}\\
We completely determine cohomology groups of sections of homogeneous
line bundles over a toroidal group. 
\ 
\footnote{
{\bf keywords:} Cohomology groups, Homogeneous line bundles, Toroidal groups}\ 

\newtheorem{definition}{Definition}[section]

\newtheorem{remark}[definition]{Remark}

\newtheorem{lemma}[definition]{Lemma}

\newtheorem{theorem}[definition]{Theorem}

\newtheorem{proposition}[definition]{Proposition}

\newtheorem{corollary}[definition]{Corollary}

\newtheorem{example}[definition]{Example}

\section{Introduction}
A toroidal group is a connected complex Lie group without
non-constant holomorphic functions. Such a group appears
as the steinizer of a complex Lie group (\cite{morimoto1},
\cite{morimoto2}). It is well-known that a toroidal group
is commutative. Then it is isomorphic to a quotient group
${\mathbb C}^n/\Gamma $ of ${\mathbb C}^n$ by a discrete
subgroup $\Gamma $. A complex torus is a compact toroidal
group. By the Remmert-Morimoto theorem (\cite{kopfermann}
and \cite{morimoto2}) every connected commutative complex
Lie group is isomorphic to the product of copies of ${\mathbb C}$,
copies of ${\mathbb C}^{*} = {\mathbb C} \setminus \{ 0 \}$
and a toroidal group. A toroidal group plays an important
role in the study of complex Lie groups. Moreover the relations
to the number theory are known (cf. \cite{abe2}, \cite{abe3}).
\par
Let $X$ be a toroidal group. The cohomology groups
$H^p(X,{\mathcal O})$ $(p\geq 1)$ were completely determined
by Kazama (\cite{kazama}).
The next problem is to determine $H^p(X,{\mathcal O}(L))$ for
any holomorphic line bundle $L$ over $X$. If $X$ is compact, i.e. a complex
torus, then we know the cohomology groups
$H^p(X,{\mathcal O}(L))$ for any 
$L$. The general result is known as the Index theorem.
In this case we need tools which are valid for compact
K\"ahler manifolds. Unfortunately, they are not applicable to
non-compact toroidal groups. 
\par
In this paper we consider homogeneous line bundles $L$ over
a toroidal group $X = {\mathbb C}^n/\Gamma $ with
${\rm rank}\,\Gamma = n+m$. It is known that $X$ has the structure
of principal $({\mathbb C}^{*})^{n-m}$-bundle
$\mu : X \longrightarrow {\mathbb T}$ over an $m$-dimensional
complex torus ${\mathbb T}$. We determine the cohomology
groups $H^p(X,{\mathcal O}(L))$ for $p\geq 1$. The following
three cases occur when $L$ is not analytically trivial:\\
(1) $H^p(X,{\mathcal O}(L))=0$ for $p\geq 1$,\\
(2) $H^p(X,{\mathcal O}(L))\cong H^p({\mathbb T},{\mathcal O})$
for $p\geq 1$,\\
(3) $H^p(X,{\mathcal O}(L))$ is a non-Hausdorff Fr\'echet space,
then, of infinite dimension for $1\leq p \leq m$\\
(Theorem 9.1).
It seems to us that this result is the first step beyond the case
of $H^p(X,{\mathcal O})$.
When $X$ is a complex torus, we know
$H^p(X,{\mathcal O}(L))=0$ $(p\geq 1)$ for any homogeneous line
bundle $L$ over $X$ which is not analytically trivial. Our
method gives another proof of this fact. We give examples which
show that each of the above cases really occurs.\par
The paper is organized as follows. In Section 2 we state standard 
coordinates and real coordinates of ${\mathbb C}^n$ and the
relation between them. In Section 3 we collect some facts about
homogeneous line bundles. In Section 4 we introduce sheaves
${\mathcal F}^{r,s}$ and ${\mathcal F}^{r,s}(L)$, and give
Dolbeault-Kazama isomorphisms. Every $\Gamma $-periodic $C^{\infty}$
function on ${\mathbb C}^n$ which is holomorphic with respect
to the last $n-m$ variables has the Fourier expansion. In Section 5
we explain the properties of such Fourier expansions and their
derivatives. The spaces $H^0(X,{\mathcal F}^{r,s})$ and
$H^0(X,{\mathcal F}^{r,s}(L))$ are isomorphic as Fr\'echet
spaces. Using these isomorphisms, we translate $\overline{\partial}$-equations
to certain equations of $\Gamma $-periodic differential forms
in Section 6. In Section 7 we obtain formal solutions of the
above equations. Then it suffices to consider the convergence of
formal solutions. We give conditions for the convergence in
Section 8. Finally we prove the main result in Section 9. In
the last section we construct examples.

\section{Preliminaries}
Let $X = {\mathbb C}^n/\Gamma $ be a toroidal group with
${\rm rank}\,\Gamma = n+m$. We use standard coordinates
$z = (z_1, \dots ,z_n)$ of ${\mathbb C}^n$ for $X$ and a period matrix $P$ of
$X$ in the first normal form as
$$P = 
\begin{pmatrix}
I_n & S
\end{pmatrix},\quad
S = 
\begin{pmatrix}
S_1 \\
S_2
\end{pmatrix}
,$$
where $I_n$ is the unit matrix of degree $n$ and $S$ is a complex
$(n,m)$-matrix with $\det ({\rm Im}(S_1)) \not= 0$.
The matrix $S$ satisfies the irrationality condition
$$\text{
for any $\tau = (\tau _1,\dots ,\tau _n) \in
{\mathbb Z}^n \setminus \{ 0 \}$ we have
$\tau S \notin {\mathbb Z}^m$}
\leqno{(IS)}$$
in standard coordinates because $X$ is a toroidal group
(cf. \cite{abe-kopfermann}).
We write $\begin{pmatrix}
I_n & S
\end{pmatrix} = (e_1, \dots ,e_n,s_1,\dots ,s_m)$.
For $j=m+1, \dots ,n$ we set $s_j := \sqrt{-1}e_j$. Then 
$\{ e_1, \dots ,e_n,s_1,\dots ,s_n \}$ is a basis of
${\mathbb C}^n$ over ${\mathbb R}$. We take real coordinates
$t = (t_1, \dots ,t_{2n})$ of ${\mathbb C}^n$ defined by
$$\sum _{i=1}^{n}z_ie_i = \sum _{i=1}^{n} t_ie_i +
\sum _{i=1}^{n} t_{n+i}s_i.$$
A column vector $s_j$ is written as $s_j = {}^t(s_{1j},\dots ,
s_{nj})$. We set $A := (a_{ij})_{1\leq i,j \leq n}$ and
$B := (b_{ij})_{1\leq i,j \leq n}$, where $a_{ij} :=
{\rm Re}(s_{ij})$ and $b_{ij} := {\rm Im}(s_{ij})$.
We note that $a_{ij} = 0$ if $1 \leq i \leq n$ and
$m+1 \leq j \leq n$. Let $C = (c_{ij})_{1 \leq i,j \leq n}
:= B^{-1}$. Then we have $c_{ij} = \delta _{ij}$ for
$1 \leq i \leq n$ and $m+1 \leq j \leq n$.
We write $z_i = x_i + \sqrt{-1}y_i$, $i=1, \dots ,n$ as usual.
Then we have the relations
\begin{equation}
t_j = x_j - \sum _{k=1}^n \left(\sum _{\ell = 1}^n
a_{j\ell }c_{\ell k}\right) y_k\quad
\text{and}\quad t_{n+j} = \sum _{k=1}^nc_{jk}y_k
\end{equation}
for $j=1, \dots , n$. Therefore we obtain
\begin{equation}
\frac{\partial }{\partial \overline{z}_j} = \frac{1}{2}
\left\{ \frac{\partial }{\partial t_j} + \sqrt{-1}\left(
- \sum _{k=1}^n \sum _{\ell = 1}^n a_{k\ell } c_{\ell j}
\frac{\partial }{\partial t_k} + \sum _{k=1}^n c_{kj}
\frac{\partial }{\partial t_{n+k}}\right) \right\}
\end{equation}
for $j=1, \dots ,n$.\par
Let ${\mathbb T} = {\mathbb C}^m/\Lambda $ be an $m$-dimensional
complex torus, where $\Lambda $ is a discrete subgroup of
${\mathbb C}^m$ with period matrix $\begin{pmatrix} I_m & S_1
\end{pmatrix}$. The toroidal group $X$ has the structure of
principal $({\mathbb C}^{*})^{n-m}$-bundle
$\mu : X \longrightarrow {\mathbb T}$ over ${\mathbb T}$ by
the projection $(z_1, \dots ,z_n) \longmapsto (z_1, \dots ,z_m)$.

\section{Homogeneous line bundles}
For any $x$ in a toroidal group $X = {\mathbb C}^n/\Gamma $ we
define a translation $T_x : X \longrightarrow X,\ y\mapsto
y + x$.
\begin{definition}
{\rm
A holomorphic line bundle $L$ over $X$ is said to be homogeneous
if $L$ and $T_x^{*}L$ is analytically isomorphic for any $x \in X$.
}
\end{definition}
A homomorphism $\rho : \Gamma \longrightarrow {\mathbb C}^{*}$
is called a (1-dimensional) representation of $\Gamma $. Since it
is considered as a factor of automorphy, it defines a holomorphic
line bundle over $X$. It is obvious by definition that if $L$
is a holomorphic line bundle given by a representation of $\Gamma $,
then it is homogeneous. When $X$ is a complex torus, it is
well-known that for a holomorphic line bundle $L$ over $X$ the
following statements are equivalent:\\
(1) $L$ is topologically trivial.\\
(2) $L$ is given by a representation of $\Gamma $.\\
(3) $L$ is homogeneous.\\
The above equivalence does not hold for a toroidal group in general.
However we have the following proposition.

\begin{proposition}[Abe \cite{abe1}]
Let $L$ be a holomorphic line bundle over a toroidal group
$X = {\mathbb C}^n/\Gamma $. Then $L$ is homogeneous if and only if
it is given by a representation of $\Gamma $.
\end{proposition}

Throughout this paper we assume that $L$ is a homogeneous line
bundle over a toroidal group $X = {\mathbb C}^n/\Gamma $ given by
a representation $\rho : \Gamma \longrightarrow {\mathbb C}^{*}$ of
$\Gamma $. Then there exists a homomorphism $d : \Gamma 
\longrightarrow {\mathbb C}$ such that $\rho (\gamma ) = {\bf e}(
d(\gamma ))$ $(\gamma \in \Gamma )$, 
where ${\bf e}(* ) = \exp (2\pi \sqrt{-1}*)$.

\begin{lemma}
The homomorphism $d : \Gamma \longrightarrow {\mathbb C}$ is
equivalent to a homomorphism $\tilde{d} : \Gamma \longrightarrow
{\mathbb R}$ with $\tilde{d}(e_{m+j}) = 0$ for $j=1, \dots , n-m$
as summands of automorphy.
\end{lemma}

\begin{proof}
First we may assume $d(e_j) = 0$ for $j=1, \dots ,n$ because
$\rho = {\bf e}(d)$ gives a topologically trivial holomorphic
line bundle. We define a discrete subgroup $\widetilde{\Gamma }$ of
rank $2n$ by
$$\widetilde{\Gamma } := \Gamma \oplus \bigoplus _{j=1}^{n-m}
{\mathbb Z}s_{m+j}.$$
Putting $d(s_{m+j}) = 0$ for $j=1, \dots ,n-m$, we extend $d$ to
a homomorphism $d : \widetilde{\Gamma } \longrightarrow {\mathbb C}$.
Furthermore it is extended to an ${\mathbb R}$-linear mapping
$d : {\mathbb C}^n \longrightarrow {\mathbb C}$. If we set 
$k := {\rm Im}(d)$, then $k : {\mathbb C}^n \longrightarrow {\mathbb R}$
is an ${\mathbb R}$-linear mapping such that
$$k(e_j) = 0, \quad j=1, \dots , n$$
and
$$k(s_{m+j}) = k(\sqrt{-1}e_{m+j}) = 0, \quad j=1, \dots , n-m.$$
We define a ${\mathbb C}$-linear mapping $\ell : {\mathbb C}^n
\longrightarrow {\mathbb C}$ by
$$\ell (v) := k(\sqrt{-1}v) + \sqrt{-1}k(v),\quad v \in {\mathbb C}^n.$$
Set $\tilde{d}(\gamma ) := d(\gamma ) - \ell (\gamma )$ for $\gamma \in \Gamma $.
Then $\tilde{d} : \Gamma \longrightarrow {\mathbb R}$ is a
homomorphism satisfying $\tilde{d}(e_{m+j}) = 0$ for
$j=1, \dots ,n-m$. Since $\ell $ is ${\mathbb C}$-linear, $d$ and
$\tilde{d}$ are equivalent.
\end{proof}

From now on we assume that the representation $\rho = {\bf e}(d)$ is
given by a homomorphism $d : \Gamma \longrightarrow {\mathbb R}$
having the properties in Lemma 3.3.

\section{Dolbeault-Kazama isomorphisms}

Consider the structure of principal $({\mathbb C}^{*})^{n-m}$-bundle
$\mu : X \longrightarrow {\mathbb T}$ stated in Section 2. We write
standard coordinates $z=(z_1, \dots ,z_n)$ as $z=(z',z'')$, where
$z' = (z_1, \dots ,z_m)$ and $z'' = (z_{m+1}, \dots , z_n)$.
The $\overline{\partial }$-operator is decomposed as
$\overline{\partial } = \overline{\partial }_1 + \overline{\partial }_2$,
where $\overline{\partial }_1$ is the $\overline{\partial }$-operator
with respect to $z'$, and $\overline{\partial }_2$ is the one with
respect to $z''$. Let ${\mathcal F}$ be the sheaf of germs of
$C^{\infty }$ functions on $X$ which are holomorphic along the
fibers $({\mathbb C}^{*})^{n-m}$, and let ${\mathcal F}^{r,s}$ be
the sheaf of $(r,s)$-forms with respect to
$\{ dz_1, \dots ,dz_m,d\overline{z}_1, \dots , d\overline{z}_m\}$
with coefficients in ${\mathcal F}$. Similarly we denote by
${\mathcal F}(L)$ the sheaf of germs of $C^{\infty }$ functions with
valued in $L$ which are holomorphic along the fibers.
Let ${\mathcal F}^{r,s}(L)$ be the sheaf of $(r,s)$-forms
with respect to $\{ dz_1, \dots ,dz_m,d\overline{z}_1, \dots , d\overline{z}_m\}$
with coefficients in ${\mathcal F}(L)$.
\par
The following proposition is due to Kazama and Umeno (Lemma 1.1 in
\cite{kazama-umeno}).

\begin{proposition}
For any $r,s$ we have
$$H^p(X,{\mathcal F}^{r,s}(L)) = 0,\quad p\geq 1.$$
\end{proposition}

We have a resolution of ${\mathcal O}(L)$
\begin{equation}
0 \longrightarrow {\mathcal O}(L) \longrightarrow
{\mathcal F}(L)\xrightarrow{\, \overline{\partial }_1\, }
{\mathcal F}^{0,1}(L) \xrightarrow{\, \overline{\partial }_1\, }
\cdots \xrightarrow{\, \overline{\partial }_1\, }
{\mathcal F}^{0,m}(L) \longrightarrow 0.
\end{equation}
By Proposition 4.1 we obtain the Dolbeault-Kazama isomorphisms
\begin{equation}
H^p(X,{\mathcal O}(L)) \cong \frac{Z_{\overline{\partial }_1}
(X,{\mathcal F}^{0,p}(L))}{B_{\overline{\partial }_1}
(X,{\mathcal F}^{0,p}(L))}
\end{equation}
for $p\geq 1$, where we set
$$Z_{\overline{\partial }_1}
(X,{\mathcal F}^{0,p}(L)) := \{ \varphi \in
H^0(X,{\mathcal F}^{0,p}(L)) ; \overline{\partial }_1
\varphi = 0 \}$$
and $B_{\overline{\partial }_1}(X,{\mathcal F}^{0,p}(L))
:= \overline{\partial }_1H^0(X,{\mathcal F}^{0,p-1}(L))$.

\section{Fourier expansion}
We can identify $H^0(X,{\mathcal F})$ with the space of all
$\Gamma $-periodic $C^{\infty }$ functions on ${\mathbb C}^n$
which are holomorphic with respect to
$z'' = (z_{m+1},\dots ,z_n)$. Real coordinates
$t = (t_1, \dots ,t_{2n})$ are written as $t = (t',t'')$,
where $t' = (t_1,\dots ,t_{n+m})$ and $t'' = (t_{n+m+1},\dots ,t_{2n})$.
Let $f \in H^0(X,{\mathcal F})$. Since $\partial f/\partial 
\overline{z}_{m+j} = 0$ for $j = 1, \dots , n-m$, we have
the following Fourier expansion of $f$:
\begin{equation}
f(t) = \sum _{\sigma \in {\mathbb Z}^{n+m}} a^{\sigma }
\exp \left( -2\pi \sum _{i=m+1}^{n} \sigma _i t_{n+i}\right)
{\bf e}(\langle \sigma ,t'\rangle),
\end{equation}
where $a^{\sigma }$ is a complex number and
$\langle \sigma ,t'\rangle = \sum _{i=1}^{n+m}\sigma _i t_i$.
The following lemma follows from the well-known result
for Fourier coefficients of $C^{\infty }$ functions.

\begin{lemma}
Let $\{ a^{\sigma } ; \sigma \in {\mathbb Z}^{n+m}\}$ be a
sequence of complex numbers. We consider a formal series
$$\sum _{\sigma \in {\mathbb Z}^{n+m}} a^{\sigma }
\exp \left( -2\pi \sum _{i=m+1}^{n} \sigma _i t_{n+i}\right)
{\bf e}(\langle \sigma ,t'\rangle ).$$
Then the series converges to a function in
$H^0(X,{\mathcal F})$ if and only if for any $R>0$ and
any $k > 0$ we have
$$\sup \{ |a^{\sigma }| R^{\sum _{i=m+1}^n|\sigma _i|}
|\sigma |^k ; \sigma \in {\mathbb Z}^{n+m}\} < \infty .$$
\end{lemma}

We write $\sigma \in {\mathbb Z}^{n+m}$ as
$\sigma = (\sigma ', \sigma '',\sigma ''')$, where
$\sigma ' = (\sigma _1, \dots ,\sigma _m)$,
$\sigma '' = (\sigma _{m+1}, \dots ,\sigma _n)$ and
$\sigma ''' = (\sigma _{n+1}, \dots ,\sigma _{n+m})$.
If we set
\begin{equation*}
\begin{split}
f^{\sigma }(t) &:= a^{\sigma }
\exp \left( -2\pi \sum _{i=m+1}^{n} \sigma _i t_{n+i}\right)
{\bf e}(\langle \sigma ,t'\rangle )\\
& = a^{\sigma }{\bf e}(\langle \sigma ,t'\rangle + \sqrt{-1}\langle \sigma '',t''\rangle ),\\
\end{split}
\end{equation*}
then (5.1) is rewritten as
\begin{equation}
f(t) = \sum _{\sigma \in {\mathbb Z}^{n+m}}
f^{\sigma }(t).
\end{equation}
By (2.2) we obtain
\begin{equation}
\frac{\partial f^{\sigma }}{\partial \overline{z}_j}
(t) = \pi \sum _{k=1}^m c_{kj}K_{\sigma ,k}f^{\sigma }(t)
\end{equation}
for $j=1, \dots ,m$, where we set
\begin{equation}
K_{\sigma ,k}:= \sum _{\ell =1}^n \sigma _{\ell }s_{\ell k}
- \sigma _{n+k}
\end{equation}
for $k=1,\dots ,m$. We note
\begin{equation}
K_{\sigma } := (K_{\sigma ,1},\dots ,K_{\sigma ,m}) =
(\sigma ', \sigma '')S - \sigma '''
\end{equation}
for $\sigma = (\sigma ', \sigma '', \sigma ''') \in {\mathbb Z}^{n+m}$.
Put
$$\widetilde{K}_{\sigma ,j} := \pi \sum _{\ell =1}^m
c_{\ell j}K_{\sigma ,\ell }$$
for $j=1, \dots ,m$. We rewrite (5.3) as
\begin{equation}
\frac{\partial f^{\sigma }}{\partial \overline{z}_j}(t)
= \widetilde{K}_{\sigma ,j}f^{\sigma }(t)
\end{equation}
for $j=1, \dots ,m$. Let $\widetilde{K}_{\sigma } :=
(\widetilde{K}_{\sigma ,1}, \dots , \widetilde{K}_{\sigma ,m})$.
Then we have the relation
\begin{equation}
\widetilde{K}_{\sigma } = \pi K_{\sigma }(c_{jk})_{1\leq j,k
\leq m}.
\end{equation}

\section{Translation of $\overline{\partial }_1$-equations}
Let $L$ be a homogeneous line bundle over a toroidal group
$X = {\mathbb C}^n/\Gamma $ with ${\rm rank}\, \Gamma = n+m$.
It is given by a representation $\rho  = {\bf e}(d)$
of $\Gamma $, where $d : \Gamma \longrightarrow {\mathbb R}$ is
a homomorphism having the properties in Lemma 3.3. We define a
linear polynomial $a(t)$ in $t$ by
\begin{equation}
a(t) := - \sum _{i=1}^m d(e_i) t_i - \sum _{i=1}^m d(s_i) t_{n+i}.
\end{equation}
It is easy to check that
\begin{equation}
a(t+\gamma ) + d(\gamma ) - a(t) = 0
\end{equation}
for any $\gamma \in \Gamma $ and any $t \in {\mathbb R}^{2n}$.
By (2.2) we obtain $\partial a/\partial \overline{z}_i = 0$
for $i=m+1, \dots ,n$. Then we have $a \in H^0({\mathbb C}^n,
\pi ^{*}{\mathcal F})$, where $\pi : {\mathbb C}^n \longrightarrow
X$ is the projection. We set
$$F(t) := {\bf e}(a(t)).$$
\par
We can consider $H^0(X,{\mathcal F}^{r,s})$, $H^0(X,{\mathcal F}^{r,s}(L))$
and $H^p(X,{\mathcal O}(L))$ as Fr\'echet spaces in the usual manner.
Then the multiplicity by $F(t)$ gives isomorphisms between
Fr\'echet spaces
$$F : H^0(X,{\mathcal F}^{r,s}(L)) \longrightarrow
H^0(X,{\mathcal F}^{r,s}),\quad \varphi \longmapsto F\varphi $$
by (6.2). We set
$$G(t) := F(t)^{-1} = {\bf e}(- a(t)).$$
Then, for any $\varphi \in H^0(X,{\mathcal F}^{0,p}(L))$ there
exists uniquely $\phi \in H^0(X,{\mathcal F}^{0,p})$ such that
\begin{equation}
\varphi = G\phi .
\end{equation}
\par
By a straight calculation using (2.2) we obtain
\begin{equation*}
\begin{split}
-a(t) & = \frac{1}{2} \sum _{j=1}^m\left( d(e_j) -\sqrt{-1}
\sum _{i=1}^m \left(d(s_i)c_{ij} - d(e_j)\sum _{k=1}^m
a_{ik}c_{kj}\right) \right) z_j\\
& \quad + \frac{1}{2} \sum _{j=1}^m\left( d(e_j) +\sqrt{-1}
\sum _{i=1}^m \left(d(s_i)c_{ij} - d(e_j)\sum _{k=1}^m
a_{ik}c_{kj}\right) \right) \overline{z}_j.
\end{split}
\end{equation*}
Then, setting
$$\alpha _j := d(e_j) - \sqrt{-1}\sum _{i=1}^m \left( d(s_i)c_{ij}
- d(e_i) \sum _{k=1}^m a_{ik} c_{kj}\right),$$
we have
\begin{equation}
- a(t) = \frac{1}{2}\sum _{j=1}^m (\alpha _j z_j +
\overline{\alpha }_j \overline{z}_j).
\end{equation}
We denote
$$A_1 := (a_{ik})_{1\leq i,k \leq m}\quad
\text{and}\quad C_1 := (c_{kj})_{1\leq k,j \leq m}.$$
Then we have
\begin{equation*}
\begin{split}
(\alpha _1, \dots ,\alpha _m) &= 
\quad (d(e_1),\dots ,d(e_m))\\
& \quad -
\sqrt{-1}\left\{ (d(s_1), \dots ,d(s_m)) - (d(e_1), \dots ,d(e_m))A_1
\right\} C_1.\\
\end{split}
\end{equation*}
Operating $\overline{\partial }_1$ on $G(t)$, we obtain
$$\overline{\partial }_1 G =
\sum _{j=1}^m \frac{\partial G}{\partial \overline{z}_j}
d\overline{z}_j = \pi \sqrt{-1}G\sum _{j=1}^m\overline{\alpha }_j
d\overline{z}_j.$$
We set $\Phi _0 := \sum _{j=1}^m \beta _j d\overline{z}_j$, where
$\beta _j := \pi \sqrt{-1}\overline{\alpha }_j$. Then we have
\begin{equation}
\overline{\partial }_1 G = G \Phi _0.
\end{equation}
By the relation (6.3) we obtain
\begin{equation}
\overline{\partial }_1 \varphi = G(\Phi _0 \wedge \phi +
\overline{\partial }_1 \phi ).
\end{equation}
Therefore $\overline{\partial }_1 \varphi = 0$ if and only if
\begin{equation}
\Phi _0 \wedge \phi + \overline{\partial }_1 \phi = 0.
\end{equation}
\par
We assume that there exists $\eta \in H^0(X,{\mathcal F}^{0,p-1}(L))$
with $\overline{\partial }_1 \eta = \varphi $. Then we can take
$\psi \in H^0(X,{\mathcal F}^{0,p-1})$ such that
$\eta = G\psi $ and
\begin{equation}
\phi = \Phi _0 \wedge \psi + \overline{\partial }_1 \psi
\end{equation}
by (6.6). Therefore the problem to find $\eta $ with 
$\overline{\partial }_1 \eta = \varphi $ for
$\varphi \in Z_{\overline{\partial }_1}(X,{\mathcal F}^{0,p}(L))$ is
translated to the following problem.\par
{\bf Problem.} For any $\phi \in H^0(X,{\mathcal F}^{0,p})$
satisfying (6.7), does there exist $\psi \in H^0(X,{\mathcal F}^{0,p-1})$
such that the equation (6.8) holds?\\

\section{Formal solutions}
We may assume that the homogeneous line bundle $L$ is not analytically
trivial. Then we have $a(t) \not= 0$. This means 
$(\overline{\alpha }_1,\dots ,\overline{\alpha }_m) \not=
(0, \dots ,0)$. Hence we have
$$(\beta _1,\dots ,\beta _m) = \pi \sqrt{-1}
(\overline{\alpha }_1,\dots ,\overline{\alpha }_m) \not=
(0, \dots ,0).$$
Consider equations
\begin{equation}
\widetilde{K}_{\sigma ,j} + \beta _j = 0,\quad
j=1,\dots ,m.
\end{equation}
It is easily seen that (7.1) is equivalent to
\begin{equation}
K_{\sigma }C_1 + \sqrt{-1} (\overline{\alpha }_1,\dots ,
\overline{\alpha }_m) = (0,\dots ,0).
\end{equation}
Furthermore, (7.2) is equivalent to
\begin{equation}
\begin{cases}
(\sigma ', \sigma ''){\rm Re}(S) + \sigma ''' +
({\rm Im}(\alpha _1), \dots , {\rm Im}(\alpha _m))C_1^{-1} = 0,\\
(\sigma ', \sigma ''){\rm Im}(S) +
({\rm Re}(\alpha _1), \dots , {\rm Re}(\alpha _m))C_1^{-1} = 0\\
\end{cases}
\end{equation}
by (5.5). We note that (7.3) does not hold for $\sigma =0$.

\begin{lemma}
If there exists $\sigma _0 =(\sigma _0', \sigma _0'', \sigma _0''')
\in {\mathbb Z}^{n+m}\setminus \{ 0 \}$ satisfying (7.3), then
it is unique.
\end{lemma}

\begin{proof}
Suppose that $\sigma = (\sigma ', \sigma '', \sigma ''') \in
{\mathbb Z}^{n+m}\setminus \{ 0 \}$ also satisfies (7.3).
Then we have
$$(\sigma ' - \sigma _0', \sigma '' - \sigma _0'')S +
\sigma ''' - \sigma _0''' = 0.$$
By the irrationality condition (IS) we obtain
$\sigma - \sigma _0 = 0$.
\end{proof}

We define
\begin{equation*}
Z :=
\begin{cases}
{\mathbb Z}^{n+m} \setminus \{ \sigma _0 \} & \text{if 
there exists $\sigma _0$ with (7.3)},\\
{\mathbb Z}^{n+m} & \text{otherwise.}\\
\end{cases}
\end{equation*}
For any $\sigma \in {\mathbb Z}^{n+m}$ we denote by $j(\sigma )$
the smallest
integer $j$ satisfying
$$|\widetilde{K}_{\sigma , j} + \beta _{j}|
= \max \{ |\widetilde{K}_{\sigma , k} + \beta _{k}|;
k = 1, \dots , m \}.$$
We note $j(\sigma _0) = 1$ for $\sigma _0$ satisfying (7.3).\par
Every $\phi \in H^0(X,{\mathcal F}^{0,p})$ has the Fourier expansion
\begin{equation}
\phi = \sum _{\sigma \in {\mathbb Z}^{n+m}}\phi ^{\sigma },\quad
\phi ^{\sigma } = {\bf e}(\langle \sigma ,t'\rangle + \sqrt{-1}(\langle \sigma '',t''\rangle )
\phi _c^{\sigma },
\end{equation}
where $\phi _c^{\sigma }$ is a $(0,p)$-form with constant coefficients.
For any $\sigma \in {\mathbb Z}^{n+m}$ we have the unique representation
\begin{equation}
\begin{split}
\phi ^{\sigma }& = \sum _{1\leq i_1 < \dots < i_{p-1}\leq m}
\left( \phi^{\sigma }_{j(\sigma )i_1\dots i_{p-1}}
d\overline{z}_{j(\sigma )} + \sum _{j \notin \{ j(\sigma ),i_1, \dots ,i_{p-1}\}}
\phi ^{\sigma }_{ji_1\dots i_{p-1}}d\overline{z}_j\right)\\
&\quad \hspace{2cm} \wedge d\overline{z}_{i_1} \wedge \dots \wedge
d\overline{z}_{i_{p-1}},
\end{split}
\end{equation}
where
$$\phi ^{\sigma }_{k i_1\dots i_{p-1}} =
a^{\sigma }_{k i_1\dots i_{p-1}}
{\bf e}(\langle \sigma ,t'\rangle + \sqrt{-1}\langle \sigma '',t''\rangle ),\quad
a^{\sigma }_{k i_1\dots i_{p-1}} \in {\mathbb C}.$$
Here we mean that $\phi ^{\sigma }_{j(\sigma )i_1 \dots i_{p-1}} = 0$
if $j(\sigma ) \in \{ i_1, \dots ,i_{p-1}\}$. Since
\begin{equation*}
\begin{split}
\overline{\partial }_1 \phi ^{\sigma }& =
\sum _{1\leq i_1 < \dots < i_{p-1}\leq m}\left\{
\sum _{j \notin \{ j(\sigma ),i_1, \dots ,i_{p-1}\}}
\left( \frac{\partial \phi ^{\sigma }_{ji_1 \dots i_{p-1}}}
{\partial \overline{z}_{j(\sigma )}} -
\frac{\partial \phi ^{\sigma }_{j(\sigma )i_1 \dots i_{p-1}}}
{\partial \overline{z}_{j}}\right)\right.\\
& \quad \hspace{2cm}
 \times d\overline{z}_{j(\sigma )} \wedge
d\overline{z}_j
\Biggr\} 
\wedge d\overline{z}_{i_1}\wedge
\dots \wedge d\overline{z}_{i_{p-1}},\\
\end{split}
\end{equation*}
we have
\begin{equation}
\begin{split}
\Phi _0 \wedge \phi ^{\sigma } + \overline{\partial }_1\phi ^{\sigma }
& = \sum _{1 \leq i_1 < \dots < i_{p-1} \leq m}\left\{
\sum _{j \notin \{ j(\sigma ), i_1, \dots ,i_{p-1}\}}
\left( \frac{\partial \phi ^{\sigma }_{ji_1 \dots i_{p-1}}}
{\partial \overline{z}_{j(\sigma )}}\right. \right.\\
&\quad - \left.
\frac{\partial \phi ^{\sigma }_{j(\sigma )i_1 \dots i_{p-1}}}
{\partial \overline{z}_{j}}
+ \beta _{j(\sigma )}\phi ^{\sigma }_
{ji_1 \dots i_{p-1}} - \beta _{j}\phi ^{\sigma }_
{j(\sigma )i_1 \dots i_{p-1}}\right)\\
&\quad \times d\overline{z}_{j(\sigma )}
\wedge d\overline{z}_j
 + \left. \sum _{\substack{j,k \notin \{ j(\sigma ),i_1,\dots
,i_{p-1}\}\\
j \not= k}}\beta _k \phi ^{\sigma }_{ji_1\dots i_{p-1}}
d\overline{z}_k \wedge d\overline{z}_j \right\}\\
&\quad \wedge d\overline{z}_{i_1} \wedge \dots \wedge
d\overline{z}_{i_{p-1}}.\\
\end{split}
\end{equation}
It follows from (5.6) that
$$\frac{\partial \phi ^{\sigma }_{ji_1\dots i_{p-1}}}
{\partial \overline{z}_k} = \widetilde{K}_{\sigma ,k}
\phi ^{\sigma }_{ji_1\dots i_{p-1}}.$$
If $\phi $ satisfies (6.7), then
\begin{equation}
(\widetilde{K}_{\sigma ,j(\sigma )} + \beta _{j(\sigma )})
\phi ^{\sigma }_{ji_1\dots i_{p-1}} =
(\widetilde{K}_{\sigma ,j} + \beta _{j})
\phi ^{\sigma }_{j(\sigma )i_1\dots i_{p-1}}
\end{equation}
for $j\notin \{ j(\sigma ),i_1,\dots ,i_{p-1}\}$ and
\begin{equation}
\beta _k \phi ^{\sigma }_{ji_1\dots i_{p-1}} = 0
\end{equation}
for $j,k \notin \{ j(\sigma ), i_1, \dots ,i_{p-1}\}$ with
$j \not= k$.\par
Assume that there exists $\psi \in H^0(X,{\mathcal F}^{0,p-1})$
satisfying (6.8). We have the expansion
$\psi = \sum _{\sigma \in {\mathbb Z}^{n+m}}\psi ^{\sigma }$
as (7.4). Each $\psi ^{\sigma }$ has the representation
\begin{equation}
\psi ^{\sigma } = \sum _{1\leq i_1 < \dots < i_{p-1}\leq m}
\psi ^{\sigma }_{i_1 \dots i_{p-1}}d\overline{z}_{i_1}
\wedge \dots \wedge d\overline{z}_{i_{p-1}},
\end{equation}
where
$$\psi ^{\sigma }_{i_1 \dots i_{p-1}} = b^{\sigma }_{i_1\dots i_{p-1}}
{\bf e}(\langle \sigma ,t'\rangle + \sqrt{-1}\langle \sigma '',t''\rangle ),\quad
b^{\sigma }_{i_1\dots i_{p-1}}\in {\mathbb C}.
$$
Since
\begin{equation*}
\begin{split}
\overline{\partial }_1 \psi ^{\sigma } & =
\sum _{1\leq i_1 < \dots < i_{p-1}\leq m}\left(
\frac{\partial \psi ^{\sigma }_{i_1\dots i_{p-1}}}
{\partial \overline{z}_{j(\sigma )}}d\overline{z}_{j(\sigma )}
+ \sum _{j \notin \{ j(\sigma ),i_1, \dots ,i_{p-1}\}}
\frac{\partial \psi ^{\sigma }_{i_1\dots i_{p-1}}}
{\partial \overline{z}_j}d\overline{z}_j\right)\\
&\quad \hspace{2cm} \wedge d\overline{z}_{i_1} \wedge \dots \wedge
d\overline{z}_{i_{p-1}}
\end{split}
\end{equation*}
and
\begin{equation*}
\Phi _0\wedge \psi ^{\sigma } = \sum _{1\leq i_1 < \dots < i_{p-1}
\leq m} \left( \sum _{k=1}^{m}\beta _k \psi ^{\sigma }_
{i_1\dots i_{p-1}}d\overline{z}_k\right) \wedge
d\overline{z}_{i_1} \wedge \dots \wedge d\overline{z}_{i_{p-1}},
\end{equation*}
we have
\begin{equation}
\frac{\partial \psi ^{\sigma }_{i_1\dots i_{p-1}}}
{\partial \overline{z}_{j(\sigma )}} + \beta _{j(\sigma )}
\psi ^{\sigma }_{i_1 \dots i_{p-1}} =
\phi ^{\sigma }_{j(\sigma )i_1 \dots i_{p-1}}
\end{equation}
and
\begin{equation}
\frac{\partial \psi ^{\sigma }_{i_1 \dots i_{p-1}}}
{\partial \overline{z}_j} + \beta _j\psi ^{\sigma }_{i_1\dots i_{p-1}}
= \phi ^{\sigma }_{ji_1\dots i_{p-1}}
\end{equation}
for $j \notin \{ j(\sigma ), i_1, \dots ,i_{p-1}\}$.
It follows from (5.6) that (7.10) and (7.11) are equivalent to
\begin{equation}
(\widetilde{K}_{\sigma ,j(\sigma )} + \beta _{j(\sigma )})
\psi ^{\sigma }_{i_1\dots i_{p-1}} =
\phi ^{\sigma }_{j(\sigma )i_1\dots i_{p-1}}
\end{equation}
and
\begin{equation}
(\widetilde{K}_{\sigma ,j} + \beta _{j})
\psi ^{\sigma }_{i_1\dots i_{p-1}} =
\phi ^{\sigma }_{ji_1\dots i_{p-1}}
\end{equation}
for $j \notin \{ j(\sigma ), i_1, \dots ,i_{p-1}\}$, respectively.
Then for any $\sigma \in Z$ and any $i_1, \dots ,i_{p-1}$ with
$1\leq i_1 < \dots < i_{p-1} \leq m$ we set
\begin{equation}
\psi ^{\sigma }_{i_1\dots i_{p-1}} :=
\frac{\phi ^{\sigma }_{j(\sigma )i_1\dots i_{p-1}}}
{\widetilde{K}_{\sigma ,j(\sigma )} + \beta _{j(\sigma )}}.
\end{equation}
We note that $\{ \psi ^{\sigma }_{i_1\dots i_{p-1}}\}$ satisfy
(7.13) for $\{ \phi ^{\sigma }_{ji_1\dots i_{p-1}}\}$ have
the property (7.7). We define $\psi ^{\sigma }$ by (7.9). We
consider the formal sum
\begin{equation}
\psi := \sum _{\sigma \in Z} \psi ^{\sigma }.
\end{equation}
Then, if $Z = {\mathbb Z}^{n+m}$, then we have
\begin{equation}
\phi = \Phi _0 \wedge \psi + \overline{\partial }_1\psi
\end{equation}
and if $Z = {\mathbb Z}^{n+m} \setminus \{ \sigma _0\}$, then we have
\begin{equation}
\phi = \Phi _0 \wedge \psi + \overline{\partial }_1\psi
+ \phi ^{\sigma _0},
\end{equation}
where
$$\phi ^{\sigma _0} = {\bf e}(\langle \sigma _0,t'\rangle + \sqrt{-1}
\langle \sigma _0'',t''\rangle ) \phi ^{\sigma _0}_c.$$

\section{Conditions}
We define $d(L) \in {\mathbb C}^m$ for any homogeneous line
bundle $L$ over a toroidal group $X$ by
$$d(L) := \sqrt{-1}(\overline{\alpha }_1,\dots ,\overline{\alpha }_m)
C^{-1}_1.$$
Then we have
\begin{equation}
(\widetilde{K}_{\sigma ,1}+\beta _1, \dots ,
\widetilde{K}_{\sigma ,m}+\beta _m) = \pi
(K_{\sigma } + d(L))C_1.
\end{equation}
Noting $K_{\sigma } + d(L) = - \sigma ''' + d(L) \not= 0$
for any $\sigma \in Z$ with $(\sigma ',\sigma '') = (0,0)$,
we set
$$m_0 := \min \{ \| - \sigma ''' + d(L)\| ;
(0,0,\sigma ''') \in Z \}.$$
Then we have
\begin{equation}
m_0 \leq \| K_{\sigma } + d(L)\|
\end{equation}
for any $\sigma = (0,0,\sigma ''') \in Z$. We consider the following
condition $(H)_S$:\\
$(H)_S$ There exists $r>0$ such that
$$\| K_{\sigma } + d(L)\| \geq r^{-|(\sigma ', \sigma '')|}$$
for all $\sigma \in Z$ with $(\sigma ',\sigma '') \not= (0,0)$.

\begin{lemma}
The condition $(H)_S$ is equivalent to the following condition
$(H)'_S$:\\
$(H)'_S$ There exist constants $C>0$ and $a\geq 0$ such that
$$\| K_{\sigma } + d(L)\| \geq C \exp (-a |(\sigma ',\sigma '')|)$$
for all $\sigma \in Z$ with $(\sigma ',\sigma '') \not= (0,0)$.
\end{lemma}

\begin{proof}
Let $r>0$ be a constant for which the condition $(H)_S$ is satisfied.
For any $r'$ with $r' \geq r$ the condition $(H)_S$ holds.
Therefore we may assume $r\geq 1$. Then the condition $(H)'_S$
holds for $C=1$ and $a:=\log r$.\par
Conversely, we suppose that the condition $(H)'_S$ is satisfied for
some constants $C>0$ and $a\geq 0$. We take $r>0$ with
$$\log r > \max \{ a-\log C, a\}.$$
Then we have
$$r^{-\lambda } < C e^{-a\lambda }\quad (\lambda \geq 1).$$
Therefore we obtain
$$r^{-|(\sigma ',\sigma '')|} < C \exp (-a|(\sigma ',\sigma '')|)
\leq \| K_{\sigma } + d(L)\|$$
for any $\sigma \in Z$ with $(\sigma ',\sigma '') \not= (0,0)$.
\end{proof}

\begin{lemma}
The condition $(H)'_S$ is equivalent to the following
condition $(H)''_S$:\\
$(H)''_S$ There exist constants $C>0$ and $a\geq 0$ such that
$$\| K_{\sigma } + d(L)\| \geq C \exp (-a |\sigma ''|)$$
for all $\sigma \in Z$ with $(\sigma ',\sigma '') \not= (0,0)$.
\end{lemma}

\begin{proof}
Since
$$\exp (-a|(\sigma ',\sigma '')|) \leq \exp (-a|\sigma ''|),$$
the condition $(H)''_S$ obviously implies the condition $(H)'_S$.\par
We suppose that the condition $(H)'_S$ is satisfied. It is
trivial that we can take constants $C>0$ and $a\geq 0$ satisfying
the inequality in $(H)''_S$ for any $\sigma \in Z$ with
$(\sigma ',\sigma '') \not= (0,0)$ and
$\| K_{\sigma } + d(L)\| > 1$. Therefore, it suffices to consider
the condition $(H)''_S$ for the following set
$$\Sigma := \{ \sigma \in Z ; (\sigma ',\sigma '') \not= (0,0),
\| K_{\sigma } + d(L)\| \leq 1\}.$$
We write $d(L) = (d(L)_1, \dots ,d(L)_m) \in {\mathbb C}^m$.
Let $\sigma \in \Sigma $. For any $\ell = 1, \dots ,m$ we have
\begin{equation}
\begin{split}
\left| \sum _{j=1}^m \sigma _j {\rm Im}(s_{j\ell })\right|
& \leq \left| \sum _{j=1}^m \sigma _j s_{j\ell } - \sigma _{n+\ell }
\right| \\
& \leq \left| \sum _{j=1}^n \sigma _j s_{j\ell } - \sigma _{n+\ell }
+ d(L)_{\ell }\right| + |d(L)_{\ell }| +
\left| \sum _{j=m+1}^n\sigma _j s_{j\ell }\right| \\
& \leq \| K_{\sigma } + d(L) \| + \| d(L) \| +
\sum _{j=m+1}^n |s_{j\ell }| |\sigma _j| \\
&\leq 1 + \| d(L) \| + \gamma _1 |\sigma ''|,\\
\end{split}
\end{equation}
where we set
$$\gamma _1 := \max \{ |s_{j\ell }| ; 1 \leq j \leq n,
1 \leq \ell \leq m \}.$$
Since $\det ({\rm Im}(S_1)) \not= 0$, there exists $\gamma _2 > 0$
such that
\begin{equation}
\max \left\{ \left| \sum _{j=1}^m \sigma _j {\rm Im}(s_{j\ell })
\right| ; \ell = 1, \dots ,m \right\} \geq \gamma _2
\sum _{j=1}^m |\sigma _j|.
\end{equation}
By (8.3) and (8.4) we can take constants $\gamma '_1 > 0$ and
$\gamma '_2 > 0$ such that
$$\sum _{j=1}^m |\sigma _j| \leq \gamma '_1 + \gamma '_2
\sum _{j=m+1}^n |\sigma _j|.$$
Let $C > 0$ and $a \geq 0$ be constants satisfying the condition
$(H)'_S$. Then we obtain
\begin{equation*}
\begin{split}
\| K_{\sigma } + d(L)\| & \geq
C \exp (-a |(\sigma ',\sigma '')|)\\
& = C \exp \left( - a \sum _{i=1}^m |\sigma _i|\right)
\exp \left( - a \sum _{j=m+1}^n |\sigma _j|\right)\\
& \geq C \exp (- a \gamma '_1) \exp \left(
- a(1 + \gamma '_2) \sum _{j=m+1}^n |\sigma _j|\right)\\
\end{split}
\end{equation*}
for $\sigma \in \Sigma $. This finishes the proof.
\end{proof}

\section{Cohomology groups}

The following theorem is our main result.

\begin{theorem}
Let $L \longrightarrow X$ be a homogeneous line bundle over a
toroidal group $X = {\mathbb C}^n/\Gamma $ with
${\rm rank}\,\Gamma = n+m$. We assume that $L$ is not analytically
trivial. We consider the condition $(H)_S$ in Section 8 and
the set $Z$ defined after Lemma 7.1.
Then one of the following cases holds.\\                                   
(I) The case  that the condition $(H)_S$ is satisfied for $d(L)$.
\\
(i) If $Z = {\mathbb Z}^{n+m}$, then
$H^p(X,{\mathcal O}(L)) = 0$ for $p \geq 1$.\\
(ii) If $Z = {\mathbb Z}^{n+m} \setminus \{ \sigma _0\}$, then
\begin{equation*}
H^p(X,{\mathcal O}(L)) \cong H^p({\mathbb T},{\mathcal O})\quad
\text{for $p \geq 1$},
\end{equation*}
where ${\mathbb T}$ is an $m$-dimensional complex torus over which
$X$ is a principal $({\mathbb C}^{*})^{n-m}$-bundle.\\
(II) The case  that the condition $(H)_S$ is not satisfied for $d(L)$.\\
For any $p$ with
$1 \leq p \leq m$, $H^p(X,{\mathcal O}(L))$ is a non-Hausdorff
Fr\'echet space, then, of infinite dimension.
\end{theorem}

We need the following lemma for the proof of Theorem 9.1.

\begin{lemma}
Under the assumption of Theorem 9.1, 
$B_{\overline{\partial }_1}(X,{\mathcal F}^{0,p}(L))$ is a closed
subspace of $Z_{\overline{\partial }_1}(X,{\mathcal F}^{0,p}(L))$ 
if and only if the space
$$\{ \Phi _0 \wedge \psi + \overline{\partial }_1\psi ;
\psi \in H^0(X,{\mathcal F}^{0,p-1})\}$$
is a closed subspace of
$$\{ \phi \in H^0(X,{\mathcal F}^{0,p}) ;
\Phi _0 \wedge \phi + \overline{\partial }_1 \phi = 0 \}.$$
\end{lemma}

\begin{proof}
For any sequence $\{ \varphi ^{(k)}\}$ in
$H^0(X,{\mathcal F}^{0,p-1}(L))$ there exists uniquely a
sequence $\{ \phi ^{(k)}\}$ in $H^0(X,{\mathcal F}^{0,p-1})$ such
that $\varphi ^{(k)} = G\phi ^{(k)}$ and
$$\overline{\partial }_1\varphi ^{(k)} = G\left(
\Phi _0 \wedge \phi ^{(k)} + \overline{\partial }_1
\phi ^{(k)}\right).$$
Then $\{ \overline{\partial }_1 \varphi ^{(k)}\}$ converges to
some $\eta \in Z_{\overline{\partial }_1}(X,{\mathcal F}^{0,p}(L))$
if and only if
$\{ \Phi _0 \wedge \phi ^{(k)} + \overline{\partial }_1
\phi ^{(k)}\}$ converges to $\tau $ with $\eta = G\tau $.
Therefore there exists $\varphi \in H^0(X,{\mathcal F}^{0,p-1}(L))$
with $\eta = \overline{\partial }_1\varphi $ if and only if
there exists $\phi \in H^0(X,{\mathcal F}^{0,p-1})$ such that
$$\tau = \Phi _0 \wedge \phi + \overline{\partial }_1\phi .$$
\end{proof}

\noindent
{\it Proof of Theorem 9.1.}
Take any $\phi \in H^0(X,{\mathcal F}^{0,p})$ satisfying (6.7).
Then we have the formal solution $\psi = \sum _{\sigma \in Z}
\psi ^{\sigma }$ of (7.16) or (7.17), where $\psi ^{\sigma }$ has the
expression (7.9).\par
(I) We assume that the condition $(H)_S$ is satisfied. By the
definition of $j(\sigma )$ we have
$$\left\| \widetilde{K}_{\sigma } + (\beta _1,\dots ,\beta _m)
\right\| \leq \sqrt{m} \left| \widetilde{K}_{\sigma ,j(\sigma )} +
\beta _{j(\sigma )}\right|.$$
There exists a positive constant $M$ such that
$$\| K_{\sigma } + d(L)\| \leq M 
\left\| \widetilde{K}_{\sigma } + (\beta _1,\dots ,\beta _m)
\right\|
$$
by (8.1). Then we obtain
\begin{equation}
\left| \widetilde{K}_{\sigma ,j(\sigma )} + \beta _{j(\sigma )}
\right| \geq \frac{1}{\sqrt{m}M} \| K_{\sigma } + d(L)\|.
\end{equation}
By Lemmas 8.1 and 8.2 we can take $C > 0$ and $a \geq 0$ such that
\begin{equation}
\| K_{\sigma } + d(L)\| \geq C \exp (- a |\sigma ''|)
\end{equation}
for any $\sigma \in Z$ with $(\sigma ',\sigma '') \not= (0,0)$.
Since $\psi ^{\sigma }_{i_1 \dots i_{p-1}}$ is defined by
(7.14), it follows from (9.1) and (9.2) that
$$\left| \psi ^{\sigma }_{i_1\dots i_{p-1}}\right| \leq
\frac{\sqrt{m}M}{C} \exp (a|\sigma ''|) \left| \phi ^{\sigma }_
{j(\sigma )i_1\dots i_{p-1}} \right|$$
for any $\sigma \in Z$ with $(\sigma ',\sigma '') \not= (0,0)$.
Moreover we have
$$\left| \psi ^{\sigma }_{i_1\dots i_{p-1}}\right| \leq
\frac{\sqrt{m}M}{m_0} \left| \phi ^{\sigma }_{j(\sigma )i_1\dots i_{p-1}}
\right|$$
for any $\sigma \in Z$ with $(\sigma ',\sigma '') = (0,0)$ by (8.2).
Therefore the series $\psi = \sum _{\sigma \in Z}\psi ^{\sigma }$
converges by Lemma 5.1.\par
If $Z = {\mathbb Z}^{n+m} \setminus \{ \sigma _0\}$, then we have
$$\phi = \Phi _0 \wedge \psi + \overline{\partial }_1\psi
+ \phi ^{\sigma _0}.$$
Therefore $\varphi = G\phi $ is $\overline{\partial }_1$-cohomologous
to $G\phi ^{\sigma _0}$. Since
$$G\phi ^{\sigma _0} = G{\bf e}(\langle \sigma _0,t'\rangle + \sqrt{-1}
\langle \sigma _0'',t''\rangle ) \phi ^{\sigma _0}_{c},$$
we obtain
$$H^p(X,{\mathcal O}(L)) \cong H^p({\mathbb T},{\mathcal O}).$$
\par
(II) Suppose that the condition $(H)_S$ is not satisfied, but
$H^p(X,{\mathcal O}(L))$ is a Hausdorff space for some $p$ with
$1\leq p \leq m$. Then $B_{\overline{\partial }_1}(X,{\mathcal F}^
{0,p}(L))$ is a closed subspace of $Z_{\overline{\partial }_1}
(X,{\mathcal F}^{0,p}(L))$. By Lemmas 8.1 and 8.2 the condition
$(H)''_S$ is not satisfied. Then, for any $\nu \in {\mathbb N}$
there exists $\sigma (\nu ) = (\sigma (\nu )',\sigma (\nu )'',
\sigma (\nu )''') \in Z$ with $(\sigma (\nu )',\sigma (\nu )'')
\not= (0,0)$ such that
$$\| K_{\sigma (\nu )} + d(L)\| < \frac{1}{\nu }
\exp ( - \nu |\sigma (\nu )''|).$$
We set
\begin{equation*}
\delta ^{\sigma } :=
\begin{cases}
\frac{\exp (- \nu |\sigma (\nu )''|)}
{\widetilde{K}_{\sigma (\nu ),j(\sigma (\nu ))} +
\beta _{j(\sigma (\nu ))}} & \text{if $\sigma = \sigma (\nu )$
for some $\nu $},\\
0 & \text{otherwise.}\\
\end{cases}
\end{equation*}
For any $\sigma \in Z$ we take $i_1,\dots ,i_{p-1}$ with
$1\leq i_1 < \dots < i_{p-1}\leq m$ such that
$j(\sigma ) \notin \{ i_1, \dots ,i_{p-1}\}$, and we define 
$$\psi ^{\sigma } := \psi ^{\sigma }_{i_1\dots i_{p-1}}
d\overline{z}_{i_1} \wedge \dots \wedge d\overline{z}_{i_{p-1}},$$
where
$$\psi ^{\sigma }_{i_1\dots i_{p-1}} = \delta ^{\sigma }
{\bf e}(\langle \sigma ,t'\rangle + \sqrt{-1}\langle \sigma '', t''\rangle ).$$
We consider the formal sum $\psi := \sum _{\sigma \in Z}\psi ^{\sigma }$.
Since
$$\frac{\exp (- \nu |\sigma (\nu )''|)}
{\left| \widetilde{K}_{\sigma (\nu ),j(\sigma (\nu ))} +
\beta _{j(\sigma (\nu ))}\right|} > \nu ,$$
$\psi $ does not converge. On the other hand, we have
\begin{equation*}
\begin{split}
\lefteqn{\Phi _0 \wedge \psi ^{\sigma } + \overline{\partial }_1
\psi ^{\sigma }}\\
& = \left( \widetilde{K}_{\sigma ,j(\sigma )} + \beta _{j(\sigma )}
\right) \psi ^{\sigma }_{i_1\dots i_{p-1}} d\overline{z}_{j(\sigma )}
\wedge d\overline{z}_{i_1} \wedge \dots \wedge
d\overline{z}_{i_{p-1}}\\
& \quad +
\sum _{j \notin \{ i_1,\dots ,i_{p-1}\}}
\left( \widetilde{K}_{\sigma ,j} + \beta _j\right)
\psi ^{\sigma }_{i_1\dots i_{p-1}} d\overline{z}_j \wedge
d\overline{z}_{i_1} \wedge \dots \wedge d\overline{z}_{i_{p-1}}.
\end{split}
\end{equation*}
Since
\begin{equation*}
\begin{split}
\left| \left( \widetilde{K}_{\sigma (\nu ),j} + \beta _j\right)
\delta ^{\sigma (\nu )}\right| &\leq
\left| \left( \widetilde{K}_{\sigma (\nu ),j(\sigma (\nu ))} + 
\beta _{j(\sigma (\nu ))}\right)
\delta ^{\sigma (\nu )}\right| \\
& = \exp ( - \nu |\sigma (\nu )''|),\\
\end{split}
\end{equation*}
$\sum _{\sigma \in Z}(\Phi _0 \wedge \psi ^{\sigma } +
\overline{\partial }_1\psi ^{\sigma })$ converges. By Lemma 9.2 there
exists $\lambda \in H^0(X,{\mathcal F}^{0,p-1})$ such that
if we set
$$\phi := \sum _{\sigma \in Z} (\Phi _0 \wedge \psi ^{\sigma } +
\overline{\partial }_1\psi ^{\sigma }),$$
then $\phi = \Phi _0 \wedge \lambda + \overline{\partial }_1\lambda $.
We express $\phi $ as in (7.4) and (7.5). We expand $\lambda $ as
$\lambda = \sum _{\sigma \in {\mathbb Z}^{n+m}}\lambda ^{\sigma }$,
where
$$\lambda ^{\sigma } = \sum _{1\leq i_1 < \dots < i_{p-1}\leq m}
\lambda ^{\sigma }_{i_1 \dots i_{p-1}}d\overline{z}_{i_1}
\wedge \dots \wedge d\overline{z}_{i_{p-1}},$$
$$\lambda ^{\sigma }_{i_1 \dots i_{p-1}} =
b^{\sigma }_{i_1\dots i_{p-1}} {\bf e}(\langle \sigma ,t'\rangle + \sqrt{-1}
\langle \sigma '',t''\rangle ),\quad
b^{\sigma }_{i_1\dots i_{p-1}}\in {\mathbb C}.$$
For any $\sigma \in Z$ we have
\begin{equation}
\Phi _0 \wedge \psi ^{\sigma } + \overline{\partial }_1
\psi ^{\sigma } = \Phi _0 \wedge \lambda ^{\sigma } +
\overline{\partial }_1\lambda ^{\sigma }.
\end{equation}
We also have
\begin{equation*}
\begin{split}
\lefteqn{\Phi _0 \wedge \lambda ^{\sigma } +
\overline{\partial }_1\lambda ^{\sigma }}\\
& =
\sum _{1\leq i_1 < \dots < i_{p-1} \leq m}
\left\{ \left( \widetilde{K}_{\sigma ,j(\sigma )} +
\beta _{j(\sigma )}\right) \lambda ^{\sigma }_{i_1\dots i_{p-1}}
d\overline{z}_{j(\sigma )}\right.\\
& \quad \hspace{1.5cm} + \sum _{j \notin \{ j(\sigma ),i_1,\dots ,i_{p-1}\}}
\left. \left( \widetilde{K}_{\sigma ,j} +
\beta _{j}\right) \lambda ^{\sigma }_{i_1\dots i_{p-1}}
d\overline{z}_{j}\right\}
 \wedge d\overline{z}_{i_1} \wedge \dots \wedge
d\overline{z}_{i_{p-1}}.\\
\end{split}
\end{equation*}
Comparing the coefficients of 
$d\overline{z }_{j(\sigma )}\wedge d\overline{z}_{i_1}
\wedge \dots \wedge d\overline{z}_{i_{p-1}}$ in both sides of
(9.3), we obtain
\begin{equation*}
\begin{split}
\lefteqn{ \left( \widetilde{K}_{\sigma ,j(\sigma )} +
\beta _{j(\sigma )}\right)\delta ^{\sigma }}\\
& =
\left( \widetilde{K}_{\sigma ,j(\sigma )} +
\beta _{j(\sigma )}\right)b^{\sigma }_{i_1\dots i_{p-1}}
+ \sum _{k=1}^{p-1} (-1)^{k}
\left( \widetilde{K}_{\sigma ,i_k} +
\beta _{i_k}\right)
b^{\sigma }_{i_1\dots i_{k-1}j(\sigma )i_{k+1}\dots i_{p-1}}.\\
\end{split}
\end{equation*}
Then we have
$$\delta ^{\sigma } = b^{\sigma }_{i_1\dots i_{p-1}} +
\sum _{k=1}^{p-1} (-1)^k \frac{\widetilde{K}_{\sigma ,i_k}+
\beta _{i_k}}{\widetilde{K}_{\sigma ,j(\sigma )} + \beta _{j(\sigma )}}
b^{\sigma }_{i_1\dots i_{k-1}j(\sigma )i_{k+1}\dots i_{p-1}}.$$
Since
$$\left|
\frac{\widetilde{K}_{\sigma ,i_k}+
\beta _{i_k}}{\widetilde{K}_{\sigma ,j(\sigma )} + \beta _{j(\sigma )}}
\right| \leq 1$$
and $\sum _{\sigma \in {\mathbb Z}^{n+m}}\lambda ^{\sigma }$ converges,
$\sum _{\sigma \in Z} \psi ^{\sigma }$ must converge.
This is a contradiction.
\hfill $\Box$

\section{Examples}

In this section we give examples which show that each of the cases
in Theorem 9.1 really occurs.\par
{\bf Example 10.1.}
Let $\Gamma $ be a discrete subgroup of rank 3 in ${\mathbb C}^2$
whose period matrix is
$$P =
\begin{pmatrix}
1 & 0 & \sqrt{-1} \alpha \\
0 & 1 & \sqrt{-1}\\
\end{pmatrix},
\quad \alpha \in {\mathbb R} \setminus {\mathbb Q}.$$
Then $X = {\mathbb C}^2/\Gamma $ is a toroidal group for the irrationality
condition (IS) is fulfilled. Let $z=(z_1,z_2)$ be standard
coordinates of ${\mathbb C}^2$ with respect to $P$.
We write $P = (e_1, e_2,s_1)$, where $s_1 = {}^t(s_{11}, s_{21})
= {}^t(\sqrt{-1}\alpha , \sqrt{-1})$. 
We set $s_2 = {}^t(s_{12}, s_{22}) := {}^t(0, \sqrt{-1})$.
Let $t = (t_1,t_2,t_3,t_4)$ be real coordinates of ${\mathbb C}^2$
with respect to $\{ e_1,e_2,s_1,s_2 \}$. In this case the
matrices $A,B$ and $C$ in Section 2 are as follows:
$$A = ({\rm Re}(s_{ij}))=0,\quad
B = ({\rm Im}(s_{ij})) =
\begin{pmatrix}
\alpha & 0 \\
1 & 1 \\
\end{pmatrix}
\quad \text{and}\quad
C = 
\begin{pmatrix}
\frac{1}{\alpha } & 0 \\
- \frac{1}{\alpha } & 1\\
\end{pmatrix}.$$
For any $\sigma = (\sigma _1,\sigma _2,\sigma _3) \in {\mathbb Z}^3$
we have
$$K_{\sigma ,1} = \sqrt{-1}(\sigma _1\alpha + \sigma _2)
- \sigma _3.$$
Furthermore we have
$K_{\sigma } = K_{\sigma ,1}$ and
$\widetilde{K}_{\sigma } = \widetilde{K}_{\sigma ,1} = \pi
K_{\sigma } \alpha $.\par
We consider a homomorphism $d : \Gamma \longrightarrow {\mathbb R}$
such that $d(e_1) = d(e_2) = 0$ and
$d(s_1) \in {\mathbb R} \setminus {\mathbb Z}$. Let $L$ be
the homogeneous line bundle over $X$ given by a representation
$\rho (\gamma ) = {\bf e}(d(\gamma ))$ $(\gamma \in \Gamma )$.
In this case we have $\alpha _1 = - \sqrt{-1}(d(s_1)/\alpha )$,
$\beta _1 = - (\pi d(s_1))/\alpha $ and $d(L) = - d(s_1)$.
Then we obtain
$$K_{\sigma } + d(L) = \sqrt{-1}(\sigma _1 \alpha + \sigma _2)
- (\sigma _3 + d(s_1)).$$
Since the equality $\widetilde{K}_{\sigma } + \beta _1 = 0$ is
equivalent to $K_{\sigma } + d(L) = 0$, we have $Z = {\mathbb Z}^3$.
We note
$$ |K_{\sigma } + d(L) | \geq |\sigma _1 \alpha + \sigma _2 |.$$
We assume that $\alpha $ is an algebraic number. By Liouville's
theorem we see at once that the condition $(H)_S$ is satisfied.
This is an example for the case (i) in (I) in Theorem 9.1.\par
{\bf Example 10.2.}
We consider a period matrix
$$P = 
\begin{pmatrix}
1 & 0 & \sqrt{-1}\\
0 & 1 & \alpha \\
\end{pmatrix},\quad
\alpha \in {\mathbb R}\setminus {\mathbb Q}.$$
It is easily seen that $X = {\mathbb C}^2/\Gamma $ is a toroidal group,
where $\Gamma $ is the discrete subgroup of ${\mathbb C}^2$ given by
$P$. As in Example 10.1 we write $P = (e_1, e_2, s_1)$.
Take a homomorphism $d : \Gamma \longrightarrow {\mathbb R}$ such that
$d(e_1) = d(e_2) = 0$ and $d(s_1) \in {\mathbb R} \setminus {\mathbb Z}$.
Then a representation $\rho = {\bf e}(d)$ of $\Gamma $ defines a
homogeneous line bundle $L$ over $X$. We obtain
$$K_{\sigma } + d(L) = \sqrt{-1}\sigma _1 + \sigma _2 \alpha
- \sigma _3 - d(s_1)$$
by a straight calculation.\par
We assume that $\alpha $ is an algebraic number. We set
$d(s_1) := \alpha $. By Liouville's theorem there exist 
$C > 0$ and $N \in {\mathbb N}$ such that
$$|q \alpha - p| \geq \frac{C}{|q|^N}$$
for $p,q \in {\mathbb Z}$ with $q \not= 0$. Then we obtain
\begin{equation*}
\begin{split}
|K_{\sigma } + d(L)| & \geq
|(\sigma _2 - 1)\alpha - \sigma _3|\\
& \geq
\frac{C}{|\sigma _2 - 1|^N}\\
& \geq
C e^{-N} \exp ( - N |\sigma _2|)
\end{split}
\end{equation*}
for any $\sigma \in {\mathbb Z}^3$ with $\sigma _2 \not= 1$.
This is the condition $(H)''_S$.
Moreover, for $\sigma _0=(0,1,0)$ we have
$$\widetilde{K}_{\sigma _0} + \beta _1 = \pi (K_{\sigma _0} +
d(L)) = 0.$$
\par
{\bf Example 10.3.}
The following example is a modification of Vogt's
example (\cite{vogt}). In Example 10.2 we set
$$\alpha = d(s_1) := \sum _{j=1}^{\infty }
10^{(-10^{j!})}.$$
For any $\nu \in {\mathbb N}$ we define
$$q_{\nu } := 10^{\nu !}10^{10^{\nu !}}\quad
\text{and}\quad
p_{\nu } := 10^{\nu !}10^{10^{\nu !}}\left(
\sum _{j=1}^{\nu } 10^{(-10^{j!})}\right).$$
Then there exists $C > 0$ such that
$$|q_{\nu } \alpha - p_{\nu }| \leq
C \exp ( - q_{\nu }^2)$$
for sufficiently large $\nu $. We set
$$\sigma ^{(\nu )} = (\sigma ^{(\nu )}_1, \sigma ^{(\nu )}_2,
\sigma ^{(\nu )}_3) := (0, q_{\nu } + 1, p_{\nu })$$
for $\nu \in {\mathbb N}$. We suppose that the condition
$(H)_S$ is satisfied. By Lemmas 8.1 and 8.2 we can take constants
$\widetilde{C} > 0$ and $\widetilde{a} \geq 0$ such that
$$|K_{\sigma } + d(L)| \geq \widetilde{C}
\exp ( - \widetilde{a} |\sigma _2|)$$
for $\sigma \in Z$ with $(\sigma _1,\sigma _2) \not= (0,0)$.
Especially the above inequality holds for $\sigma ^{(\nu )}$.
On the other hand, we have
$$
|K_{\sigma ^{(\nu )}} + d(L)| = |\sigma ^{(\nu )}_2\alpha
- \sigma ^{(\nu )}_3 - \alpha |
= |q_{\nu }\alpha - p_{\nu }|.
$$
Then we obtain
$$0 < \frac{\widetilde{C}}{C} \leq
\exp \left(- q_{\nu }^2 + \widetilde{a}(q_{\nu } + 1)\right)$$
for sufficiently large $\nu $, which is impossible.
Thus we conclude that the condition $(H)_S$ is not satisfied.





\flushleft{
Graduate School of Science and Engineering for Research\\
University of Toyama\\
Toyama 930-8555, Japan\\}

\noindent
e-mail: abe@sci.u-toyama.ac.jp\\

\end{document}